\documentclass[11pt,reqno]{amsart}

\usepackage[T1]{fontenc}
\usepackage{lmodern}
\usepackage{amsmath,amssymb,mathtools}
\usepackage{microtype}
\usepackage{xcolor}
\usepackage[colorlinks=true,linkcolor=blue,citecolor=blue,urlcolor=blue]{hyperref}
\usepackage{booktabs}
\usepackage{array}
\usepackage{tabularx}
\usepackage{caption}

\newtheorem{theorem}{Theorem}[section]
\newtheorem{proposition}[theorem]{Proposition}
\newtheorem{lemma}[theorem]{Lemma}

\newtheorem{problem}[theorem]{Problem}

\newcommand{\Z}{\mathbb{Z}}

\newcommand{\grammarrev}[1]{#1}

\title[Finite Sidon Sets via Vector-Valued Smoothing]{An Improved Upper Bound for Finite Sidon Sets via Vector-Valued Smoothing\footnote{Research supported by National Key R\&D Program of China (Grant No. 2023YFA1010202). }}
\author{Jianfeng Hou}
\address{Center for Discrete Mathematics, Fuzhou University,
Fujian 350108, China}
\email{jfhou@fzu.edu.cn}

\author{Hongbin Zhao}
\address{Center for Discrete Mathematics, Fuzhou University,
Fujian 350108, China}
\email{hbzhao2024@163.com}


\begin{document}

\begin{abstract}
Let $F(N)$ denote the largest cardinality of a Sidon subset of
$\{0,1,\ldots,N-1\}$. We prove
\[
F(N)\le N^{1/2}+\gamma_0N^{1/4}+O(1),
\qquad
\gamma_0=0.94349\ldots<0.9435.
\]
This improves the previously published coefficient $0.98183$. Our argument develops a vector-valued smoothing method that
combines several discrete smoothing kernels, each accompanied by boundary
weights that compensate for endpoint effects, so that their weighted
combination satisfies the required finite covering inequalities. We also show
that averaging systems that satisfy these inequalities individually cannot
improve upon the best constituent. Numerical optimization is used to find an
eight-component candidate system, which is then certified by exact rational
arithmetic.
\end{abstract}

\maketitle

\section{Introduction}\label{sec:introduction}

A finite set $A\subset\Z$ is \emph{Sidon} if every nonzero difference has at most one ordered representation, that is,
\[
a-b=c-d\ne0,\qquad a,b,c,d\in A,
\]
implies $(a,b)=(c,d)$. Write
\[
F(N):=\max\bigl\{|A|:A\subset\{0,1,\ldots,N-1\}\text{ is Sidon}\bigr\}.
\]

A natural problem is to understand the asymptotic behavior of $F(N)$ as $N\to\infty$. For the lower bound, Sidon first showed that $F(N)> cN^{1/2}$ \cite{Sidon1932} for some constant $c>0$, and
Erd\H{o}s and Tur\'an later improved the constant \cite{ErdosTuran1941}, proving
\[
    F(N)\geq
    \left(\frac{1}{\sqrt{2}}-o(1)\right)\sqrt{N}. 
\]
The finite-field constructions of Singer
\cite{Singer1938}  and Bose \cite{Bose1942}, together with Chowla's
transfer argument \cite{Chowla1944}, imply that
$F(N)\ge(1-o(1))\sqrt N$.  Using the prime-gap theorem of Baker,
Harman, and Pintz~\cite{BakerHarmanPintz2001}, one further obtains 
\[
F(N)\ge N^{1/2}-O\bigl(N^{21/80}\bigr).
\]

For the upper bound, observe that the positive differences determined by a Sidon set $A$ are pairwise distinct and all lie in
$\{1,\ldots,N-1\}$. \grammarrev{Consequently, $\binom{|A|}{2}\leq N-1$,
which gives $F(N)\leq \sqrt{2N}+O(1)$.} Erd\H{o}s and Tur\'an \cite{ErdosTuran1941} substantially improved
this elementary estimate by proving
\begin{equation}\label{Erdos-Turan-upper-bound-sidon}
    F(N)\leq N^{1/2}+N^{1/4}+O(1).
\end{equation}
They also posed the following problem, which remains open.
\begin{problem}[Erd\H{o}s--Tur\'an]\label{Erdos-Turan-problem}
Is it true that
\[
    F(N)\leq \sqrt{N}+O(1)
\]
for all positive integers $N$?
\end{problem}
Lindstr\"om \cite{Lindstrom} later recovered \eqref{Erdos-Turan-upper-bound-sidon} by a different argument.
Erd\H{o}s~\cite{Erdos1981,Erdos1982,Erdos1989} subsequently returned to this problem on several occasions.  He also offered \$500 for a proof or disproof of the following weaker version of Problem~\ref{Erdos-Turan-problem} \cite{Erdos1994}: for every $\varepsilon>0$,
\[
F(N)<N^{1/2}+o\bigl(N^\varepsilon\bigr)
\]
as $N\to\infty$.
The problem was later recorded in Guy's monograph \cite{Guy2004}.
Ben Green's collection of 100 open problems \cite{Green100Problems} includes as Problem~31 the question of improving the classical bounds for $F(N)$, at least for infinitely many $N$.

Since then, progress on the upper bound has mainly concerned the coefficient of the second-order term $N^{1/4}$.  Balogh, F\"uredi, and Roy \cite{BFR} reduced this coefficient to $0.998$, and O'Bryant \cite{OBryantSize} subsequently obtained $0.99703$.  Carter, Hunter, and O'Bryant \cite{CHO} further improved the bound to
\[
    F(N)\leq N^{1/2}+0.98183N^{1/4}+O(1).
\]
More recently, Tao reported unpublished joint work with Carter, Georgiev, G\'omez-Serrano, Hunter, O'Bryant, and Wagner that reduces the coefficient to $0.97633$ \cite{OptimizationConstants,ErdosForum30}. In the same forum comment, Tao mentioned a tentative value of approximately $0.947$, subject to confirmation~\cite{ErdosForum30}.

In this paper, we obtain the following further improvement.

\begin{theorem}\label{thm:main}
There is a constant $\gamma_0=0.94349\ldots<0.9435$ such that
\[
F(N)\le N^{1/2}+\gamma_0N^{1/4}+O(1).
\]
\end{theorem}

The contribution is a method rather than merely a new numerical coefficient.
Theorem~\ref{thm:kernel-transfer} converts any feasible kernel--boundary system into a Sidon-set bound; the optimized constant is therefore the output of a reusable finite optimization problem. The genuinely vector-valued feature is that no individual kernel--boundary pair need satisfy the finite covering inequalities; only their weighted combination must do so.  Proposition~\ref{prop:scalar-bottleneck} shows, conversely, that merely averaging scalar-feasible kernel--boundary pairs cannot improve the best constituent coefficient.  Thus the possible gain comes from complementary boundary deficits, not from averaging constants.

There are three logically distinct layers.  The analytic layer proves the vector-valued smoothing theorem.  The exploratory layer carries out a nonconvex outer search numerically in order to find a promising kernel--boundary system.  The certification layer rounds that system to rational data and verifies every hypothesis of the analytic theorem by exact arithmetic.  Only the first and third layers enter the proof of Theorem~\ref{thm:main}; no assertion of global optimality for the numerical search is needed.

The paper is organized as follows. Section~\ref{sec:smoothing} establishes the vector-valued smoothing theorem. Section~\ref{sec:optimization} formulates the finite optimization problem and describes the numerical search and the rationalization procedure.
Section~\ref{sec:certificate} gives the exact eight-pair rational certificate and completes the proof.

\section{Vector-valued smoothing theorem}\label{sec:smoothing}

Fix integers $m,L,R\ge1$. For $1\le r\le R$, let $p^{(r)}=\bigl(p_0^{(r)},\ldots,p_{m-1}^{(r)}\bigr)$ be a symmetric probability vector, which we call a \emph{kernel vector}, that is, 
\[
p_i^{(r)}\ge0,\qquad
\sum_{i=0}^{m-1}p_i^{(r)}=1,\qquad
p_i^{(r)}=p_{m-1-i}^{(r)}.
\]
Let $\lambda_1,\ldots,\lambda_R\ge0$ be mixing weights with $\sum_{r=1}^R\lambda_r=1$. For every $r$, let $w^{(r)}=\bigl(w_0^{(r)},\ldots,w_{Lm-1}^{(r)}\bigr)$ be a boundary vector with real entries, and set $w_j^{(r)}=1$ for every $j\ge Lm$. We call \( \left(p^{(r)},w^{(r)}\right) \) a \emph{kernel--boundary pair}, 
and the collection \( \left(p^{(r)},w^{(r)}\right)_{r=1}^R \) a \emph{kernel--boundary system}. We define the \emph{finite covering inequalities}
\begin{equation}\label{eq:covering}
\sum_{r=1}^R\left(\lambda_r\sum_{i=0}^{m-1}p_i^{(r)}w_{q+i}^{(r)}\right)\ge1
\qquad \text{ for } \quad 0\le q\le Lm.
\end{equation}
Define
\begin{equation}\label{eq:a}
a:=m\sum_{r=1}^R\lambda_r\sum_{i=0}^{m-1}\left(p_i^{(r)}\right)^2,
\end{equation}
\begin{equation}\label{eq:b}
b:=1+2\left(\frac1m\sum_{r=1}^R\lambda_r
\sum_{j=0}^{Lm-1}\left(w_j^{(r)}\right)^2-L\right).
\end{equation}
Here $a$ records the \emph{kernel energy} and $b$ the \emph{boundary cost} of the system $ \left(p^{(r)},w^{(r)}\right)_{r=1}^R$. 

\begin{theorem}\label{thm:kernel-transfer}
If \eqref{eq:covering} holds and $b>0$, then for sufficiently large $N$, 
\[
F(N)\le N^{1/2}+\sqrt{ab}\,N^{1/4}+O(1).
\]
\end{theorem}

\begin{proof}
For sufficiently large $N$, let $A\subset\{0,1,\ldots,N-1\}$ be a Sidon set and set $k=|A|$. Define the indicator function of $A$,  $\mathbf 1_A:\mathbb Z\to\{0,1\}$ by $\mathbf 1_A(n)=1$ if $n\in A$, and $\mathbf 1_A(n)=0$ otherwise.

We begin with the standard device of smoothing $\mathbf 1_A$ by convolution with a finitely supported probability kernel.
Here and below, a probability kernel on $\mathbb Z$ means a nonnegative sequence of total mass one. Let $h$ be a positive integer and set $H=mh$. For each $r$, define a finitely supported probability kernel $K^{(r)}=(K_s^{(r)})_{s\in\mathbb Z}$ by 
\begin{align*}
    K_s^{(r)}:=
\begin{cases}
 \dfrac{p_i^{(r)}}{h},& \text{if } ih\le s<(i+1)h  \text{ for some } 0\le i<m,\\[5pt]
 0, & \text{otherwise}.
\end{cases}
\end{align*}
The symmetry of $p^{(r)}$ gives
$K_s^{(r)}=K_{H-1-s}^{(r)}$. For each \(r\), define \(u_r := \mathbf{1}_A * K^{(r)}\). That is, for \(n \in \mathbb{Z}\),
\[
u_r(n) := (\mathbf{1}_A * K^{(r)})(n) = \sum_{x \in A} K^{(r)}_{n-x}.
\]

Consider the weighted energy
\[
\mathcal E:=\sum_{r=1}^R\lambda_r\|u_r\|_2^2.
\]
We first obtain an upper bound of $\mathcal E$ from the Sidon set property. Let
\[
C_r(d):=\sum_{s\in\Z}K_s^{(r)}K_{s+d}^{(r)}
\]
be the autocorrelation of $K^{(r)}$. 
Clearly, $C_r$ is  nonnegative.
A change of variables gives
\[
C_r(-d) = \sum_{s\in\mathbb Z} K_s^{(r)} K_{s-d}^{(r)} = \sum_{t\in\mathbb Z} K_{t+d}^{(r)} K_t^{(r)} = C_r(d),
\]
so \(C_r\) is even. Moreover, the normalization of $K^{(r)}$  gives
\[
\sum_{d\in\mathbb Z} C_r(d) = \sum_{s\in\mathbb Z} K_s^{(r)} \sum_{d\in\mathbb Z} K_{s+d}^{(r)} = \left( \sum_{s\in\mathbb Z} K_s^{(r)} \right)^2 = 1.
\]

Since $A$ is Sidon, the positive differences $a-b$, with $a,b\in A$ and $a>b$, are pairwise distinct. Put $\Delta(A):=\{a-b:a,b\in A,\ a>b\}$.
Hence
\begin{align*}
\|u_r\|_2^2
&=\sum_{n\in\mathbb Z}|u_r(n)|^2
=\sum_{n\in\mathbb Z}  \left(\sum_{a\in A}K^{(r)}_{n-a}\right) \left(\sum_{b\in A}K^{(r)}_{n-b}\right)\\
&= \sum_{a,b\in A}\sum_{n\in\mathbb Z}K^{(r)}_{n-a}K^{(r)}_{n-b} =\sum_{a,b\in A}C_r(a-b)\\
&=kC_r(0)+2\sum_{d\in\Delta(A)}C_r(d) \le kC_r(0)+2\sum_{d\ge1}C_r(d)\\
&=1+(k-1)C_r(0).
\end{align*}
By the definition of $K^{(r)}$ and $H=mh$,
\[
C_r(0)=\sum_s\left(K_s^{(r)}\right)^2
=\frac mH\sum_{i=0}^{m-1}\left(p_i^{(r)}\right)^2.
\]
Averaging with respect to the mixing weights $\lambda_r$, and using $\sum_{r=1}^R\lambda_r=1$, we obtain
\begin{align}\label{eq:energy-upper}
\mathcal{E}&=\sum_{r=1}^R\lambda_r\|u_r\|_2^2
 \le1+\frac{k-1}{H}m\sum_{r=1}^R\lambda_r \sum_{i=0}^{m-1}\bigl(p_i^{(r)}\bigr)^2
 =1+\frac{a(k-1)}{H}.
\end{align}

We now turn to the lower bound. The support of every $u_r$ is contained in
\[
J:=\{0,1,\ldots,N+H-2\}.
\]
For the choice of $H$ made at the end of the proof, we have $H=O(N^{3/4})=o(N)$. Thus, since $L$ is fixed, we may assume throughout that $N\ge 2LH$.

We next construct auxiliary weight functions that enforce the required covering inequality near the two endpoints while remaining equal to $1$ in the interior. Define the left and right boundary regions by $B_{\mathrm L}:=\{0,\ldots,LH-1\}$ and $B_{\mathrm R}:=\{N+H-1-LH,\ldots,N+H-2\}$.
The assumption $N\ge 2LH$ ensures that these regions are disjoint.
For $0\le x<N$, the discrete interval $I_x:=\{x,x+1,\ldots,x+H-1\}$ meets $B_{\mathrm L}$ only if $x<LH$ and meets $B_{\mathrm R}$ only if $x\ge N-LH$. In particular, it meets at most one boundary region.
For each $r$, define $Q_r:J\to\mathbb R$ by
\[
Q_r(n)=
\begin{cases}
w_j^{(r)},
& jh\le n<(j+1)h,\quad 0\le j<Lm,\\
w_j^{(r)},
& jh\le N+H-2-n<(j+1)h,\quad 0\le j<Lm,\\
1,
& \text{otherwise}.
\end{cases}
\]

We claim that, for every $x\in\{0,\ldots,N-1\}$,
\begin{equation}\label{eq:boundary-claim}
\sum_{r=1}^R\lambda_r\sum_{s=0}^{H-1}K_s^{(r)}Q_r(x+s)\ge1.
\end{equation}
We distinguish the following three ranges:
\[
0\le x<LH,\qquad
LH\le x<N-LH,\qquad
N-LH\le x<N.
\] 

First suppose that $0\le x\le LH-1$, and write $x=qh+t$ with $0\le q<Lm$ and $0\le t<h$. The convention $w_j^{(r)}=1$ for $j\ge Lm$ accounts for the part of $I_x$ lying beyond $B_{\mathrm L}$. For the $i$th block of $K^{(r)}$, exactly $h-t$ points are assigned the boundary value $w_{q+i}^{(r)}$ and exactly $t$ points are assigned the boundary value $w_{q+i+1}^{(r)}$. Consequently the left-hand side of
\eqref{eq:boundary-claim} equals
\[
\begin{aligned}
\left(1-\frac th\right)
\sum_{r=1}^R\lambda_r\sum_{i=0}^{m-1}p_i^{(r)}w_{q+i}^{(r)} +\frac th
\sum_{r=1}^R\lambda_r\sum_{i=0}^{m-1}p_i^{(r)}w_{q+i+1}^{(r)}.
\end{aligned}
\]
This is a convex combination of the finite covering inequalities with indices $q$ and $q+1$, and hence is at least $1$ by \eqref{eq:covering}. Notice that the case $q=Lm-1$ uses the final finite covering inequality with index $Lm$.

Next suppose that $N-LH\le x\le N-1$, and write
$x=N-1-(qh+t)$ with $0\le q<Lm$ and $0\le t<h$. Set
$s'=H-1-s$. The symmetry of $K^{(r)}$ gives
$K_s^{(r)}=K_{s'}^{(r)}$, while
\[
N+H-2-(x+s)=qh+t+s'.
\]
The reflected definition of $Q_r$ therefore reduces the right boundary case
to the preceding left boundary calculation.

Finally, if
$LH\le x\le N-LH-1$, then $x+s$ lies in neither boundary region for every
$0\le s<H$. Every relevant value of $Q_r$ is $1$, so equality holds in
\eqref{eq:boundary-claim}. Indeed, in this case its left-hand side
is 
\[
\sum_{r=1}^R\lambda_r\sum_{s=0}^{H-1}K_s^{(r)}=1.
\]
This proves the claim.

Summing \eqref{eq:boundary-claim} over $x\in A$ and interchanging sums gives
\begin{align}\label{EQ:upper-bound-k}
    k
&\le
\sum_{r=1}^R \lambda_r
\sum_{x\in A}\sum_{s=0}^{H-1}
K_s^{(r)}Q_r(x+s) \notag\\
&=
\sum_{r=1}^R \lambda_r
\sum_{n\in J} Q_r(n)
\sum_{x\in A}K_{n-x}^{(r)}
=
\sum_{r=1}^R \lambda_r\langle Q_r,u_r\rangle .
\end{align}
Since $\operatorname{supp}(u_r)\subseteq J$, we may regard each $u_r$ as an element of the Hilbert space $\ell^2(J)$. Let $\mathcal H:=\bigoplus_{r=1}^R\ell^2(J)$ be the direct sum of $R$ copies of  $\ell^2(J)$, and define $\mathbf Q:=\bigl(\sqrt{\lambda_r}Q_r\bigr)_{r=1}^R$ and $\mathbf u:=\bigl(\sqrt{\lambda_r}u_r\bigr)_{r=1}^R$. 
By the definition of the direct-sum inner product, 
\begin{align}\label{EQ:inner-product-Q-u}
\langle\mathbf Q,\mathbf u\rangle_{\mathcal H}
=\sum_{r=1}^R
\bigl\langle\sqrt{\lambda_r}Q_r,
\sqrt{\lambda_r}u_r\bigr\rangle_{\ell^2(J)}
=\sum_{r=1}^R\lambda_r\langle Q_r,u_r\rangle.
\end{align}
Combining \eqref{EQ:upper-bound-k} and \eqref{EQ:inner-product-Q-u}, we have 
\[
k\le\langle\mathbf Q,\mathbf u\rangle_{\mathcal H}.
\] 
Since the latter inner product is nonnegative, we may square this inequality and apply the Cauchy--Schwarz inequality in $\mathcal H$ to obtain
\begin{align}\label{eq:cauchy}
k^2  \le \bigl|\langle\mathbf Q,\mathbf u\rangle_{\mathcal H}\bigr|^2
    \le \|\mathbf Q\|_{\mathcal H}^2\|\mathbf u\|_{\mathcal H}^2 
    =\left(\sum_{r=1}^R\lambda_r\|Q_r\|_2^2\right) \mathcal E.
\end{align}
Note that 
\begin{align*}
    \|Q_r\|_2^2
&=N+H-1-2LH
  +2h\sum_{j=0}^{Lm-1}\bigl(w_j^{(r)}\bigr)^2\\
&=N+H-1
  +2h\sum_{j=0}^{Lm-1}
       \left(\bigl(w_j^{(r)}\bigr)^2-1\right).
\end{align*}
Averaging this identity over $r$ with  $\lambda_r$ gives
\begin{equation}\label{eq:weight-energy}
\sum_{r=1}^R \lambda_r \|Q_r\|_2^2
=N+H-1+2h\sum_{r=1}^R \lambda_r\sum_{j=0}^{Lm-1}\left(\bigl(w_j^{(r)}\bigr)^2-1\right)
=N+bH-1,
\end{equation}
where the last identity follows from \eqref{eq:b}.

Combining \eqref{eq:cauchy} and \eqref{eq:weight-energy} gives
\begin{equation}\label{eq:energy-lower}
\mathcal E
=\sum_{r=1}^R\lambda_r\|u_r\|_2^2
\ge
\frac{k^2}{\sum_{r=1}^R\lambda_r\|Q_r\|_2^2}
=
\frac{k^2}{N+bH-1}.
\end{equation}
Combining this with \eqref{eq:energy-upper} gives 
\begin{equation}\label{eq:master}
k^2\le(N+bH-1)\left(1+\frac{a(k-1)}H\right).
\end{equation}

It remains to choose the smoothing scale $H$. It follows from  $\binom{k}{2}\le N-1$ that $k=O(N^{1/2})$. Moreover, the Cauchy--Schwarz inequality gives
\begin{align*}
    m\sum_{i=0}^{m-1}\left(p_i^{(r)}\right)^2 \ge \left(\sum_{i=0}^{m-1}p_i^{(r)}\right)^2=1,
\end{align*}
for every $r$. Therefore $a\ge1$. Since $b>0$, the following choice is well defined.
Choose
\[
H \coloneqq m\left\lfloor \frac{1}{m}\sqrt{\frac{a}{b}}\,N^{3/4} +\frac12 \right\rfloor.
\]
For all sufficiently large $N$,  $H=\sqrt{a/b}\,N^{3/4}+O(1)$.
Expanding \eqref{eq:master} and using $k=O(N^{1/2})$ gives
\begin{align}\label{EQ:upper-bound-k2-ab}
k^2\le{}N+bH-1+\frac{aN(k-1)}H +ab(k-1)-\frac{a(k-1)}H.
\end{align}

We  estimate each term on the right-hand side of \eqref{EQ:upper-bound-k2-ab}. The estimate for $H$  gives 
\begin{align*}
    \frac{aN}{H}=\sqrt{ab}\,N^{1/4}+O(N^{-1/2}).
\end{align*}
By the classical Erd\H{o}s--Tur\'an bound
\eqref{Erdos-Turan-upper-bound-sidon},
$k-1\le N^{1/2}+O(N^{1/4})$. Combining this with the preceding estimate,
\begin{align*}
bH =\sqrt{ab}\,N^{3/4}+O(1), \qquad
\frac{aN(k-1)}H
&\le \sqrt{ab}\,N^{3/4}+O(N^{1/2}).
\end{align*}
Moreover, $ab(k-1)=O(N^{1/2})$. Therefore \eqref{EQ:upper-bound-k2-ab} gives
\[
k^2\le N+2\sqrt{ab}\,N^{3/4}+O(N^{1/2}).
\]
Using $(1+x)^{1/2}=1+x/2+O(x^2)$ as $x\to0$, we conclude that
\begin{equation*}
    k\le N^{1/2}\left(1+2\sqrt{ab}N^{-1/4}+O(N^{-1/2})\right)^{1/2} =N^{1/2}+\sqrt{ab}\,N^{1/4}+O(1). \qedhere
\end{equation*}
\end{proof}

We remark that the key idea in our application of Theorem~\ref{thm:kernel-transfer} is to use a genuinely vector-valued kernel--boundary system, in the sense that it uses more than one positive-weight kernel--boundary pair. The exact eight-pair certificate in Section~\ref{sec:certificate} has the stronger property that all eight mixing weights are positive and every constituent kernel--boundary pair violates at least one scalar covering inequality; only their weighted combination is feasible. This should not be confused with simply averaging several kernel–boundary pairs that are each feasible when used alone, that is, in the scalar case  $R=1$.  In fact, such a mixture can never improve upon the best constituent pair used alone.  We now explain this distinction in detail. For a single kernel--boundary pair  $(p^{(r)},w^{(r)})$, we call it \emph{scalar-feasible} if it satisfies \eqref{eq:covering} by itself (with $R=1$ and $\lambda_1=1$). 
Put
\[
a_r:=m\sum_{i=0}^{m-1}\bigl(p_i^{(r)}\bigr)^2,
\qquad
b_r:=1+2\left(\frac1m\sum_{j=0}^{Lm-1}
\bigl(w_j^{(r)}\bigr)^2-L\right).
\]
For this pair used alone, $a_r$ records the kernel energy and $b_r$ the boundary cost. 
The argument in Subsection~\ref{subsec:positive-b}, applied with $R=1$, gives $b_r\ge 1/m>0$. Hence Theorem~\ref{thm:kernel-transfer} applies to $(p^{(r)},w^{(r)})$ and gives the coefficient $\sqrt{a_rb_r}$. By the definitions in \eqref{eq:a} and \eqref{eq:b}, the system quantities satisfy
\[
a=\sum_{r=1}^R\lambda_ra_r,
\qquad b=\sum_{r=1}^{R}\lambda_rb_r.
\]

\begin{proposition}
\label{prop:scalar-bottleneck}
Suppose that every pair $(p^{(r)},w^{(r)})$ is scalar-feasible.  If
$\lambda_r\ge0$,
$\sum_r\lambda_r=1$, then
\[
\sqrt{ab}\ge \min_{r:\lambda_r>0}\sqrt{a_rb_r}.
\]
\end{proposition}

\begin{proof}
The Cauchy--Schwarz inequality gives
\[
ab\ge\left(\sum_r\lambda_r\sqrt{a_rb_r}\right)^2
\ge \min_{r:\lambda_r>0}a_rb_r,
\]
which proves the proposition.
\end{proof}

We also remark that the hypothesis of Proposition~\ref{prop:scalar-bottleneck} is deliberately
stronger than the finite covering inequalities \eqref{eq:covering}.  Indeed, define the boundary deficit
\[
d_r(q):=1-\sum_{i=0}^{m-1}p_i^{(r)}w_{q+i}^{(r)}.
\]
If every kernel--boundary pair is scalar-feasible, then
\[
d_r(q)\le 0
\qquad\text{for every }r\text{ and every }0\le q\le Lm.
\]
By contrast, the finite covering inequalities in \eqref{eq:covering} require only
\[
\sum_{r=1}^R \lambda_r d_r(q)\le 0
\qquad\text{for every }0\le q\le Lm.
\]
Thus the individual deficits need not be nonpositive: only their
weighted sum must be nonpositive at each boundary index.

\section{Finite kernel-vector optimization}\label{sec:optimization}

This section develops the finite optimization needed to construct
certificates for Theorem~\ref{thm:kernel-transfer}.  The quadratic-program
identities are exact.  By contrast, the outer search in
Section~\ref{subsec:outer-search} is exploratory and has no logical role in the
final proof.

\subsection{Exact boundary optimization}\label{subsec:exact-inner}

For fixed $R,m,L$, we seek to minimize $\sqrt{ab}$ over the mixing weights $\lambda_r$, symmetric kernel vectors $p^{(r)}$ and boundary vectors $w^{(r)}$ subject to \eqref{eq:covering}. The full problem is nonconvex because the finite covering inequalities \eqref{eq:covering} contain products involving kernel-vector entries, mixing weights, and boundary-vector entries. However, once the kernel vectors and mixing weights are fixed, the boundary problem becomes a strictly convex quadratic program after all zero-weight pairs are discarded.

Put $n=Lm$. If $\lambda_r=0$, the corresponding kernel--boundary pair contributes neither to the finite covering inequalities nor to the objective. We may therefore discard every zero-weight pair and relabel the remaining pairs and mixing weights. Henceforth we assume $\lambda_r>0$ for $1\le r\le R$. 
Write the stacked boundary vector as
\[
w=(w_j^{(r)}:1\le r\le R,\ 0\le j<n)\in\mathbb R^{Rn},
\qquad
D=\operatorname{diag}(\lambda_1I_n,\ldots,\lambda_RI_n).
\]
Then $D$  is positive definite. Index the rows of the matrix $B\in\mathbb R^{n\times Rn}$ by $0\le q<n$ and its columns by pairs $(r,j)$ with $1\le r\le R$ and $0\le j<n$, ordered first by $r$ and then by $j$. Define
\[
B_{q,(r,j)}=
\begin{cases}
\lambda_rp_{j-q}^{(r)},&0\le j-q<m,\\
0,&\text{otherwise},
\end{cases}
\]
or, equivalently,
\[
B=\bigl(\lambda_1B^{(1)}\ \cdots\ \lambda_RB^{(R)}\bigr),
\qquad B^{(r)}\in\mathbb R^{n\times n},
\]
\[
B^{(r)}_{q,j}=
\begin{cases}
p_{j-q}^{(r)},&q\le j\le \min\{q+m-1,n-1\},\\
0,&\text{otherwise}.
\end{cases}
\]
Observe that each $B^{(r)}\in\mathbb R^{n\times n}$ is an upper-triangular
banded matrix, and 
\[
(Bw)_q=\sum_{r=1}^R\lambda_r
\sum_{\substack{0\le i<m\\q+i<n}}
p_i^{(r)}w_{q+i}^{(r)}
\qquad(0\le q<n).
\]
Define the residual-coverage vector
\[
c=(c_0,\ldots,c_{n-1})^{\mathsf T}\in\mathbb R^n
\]
by
\[
c_q:=1-\sum_{r=1}^R\lambda_r
\sum_{\substack{0\le i<m\\q+i\ge n}}p_i^{(r)}
\qquad(0\le q<n).
\]
With these definitions, the finite covering inequalities in \eqref{eq:covering} take the following standard matrix form: 
\[
Bw\ge c. 
\]
Indeed, for $0\le q<n$, split the left-hand side of the finite covering inequality indexed by $q$ according as $q+i<n$ or $q+i\ge n$. Since $w_{q+i}^{(r)}=1$ in the latter range, we obtain
\[
\begin{aligned}
\sum_{r=1}^R\lambda_r\sum_{i=0}^{m-1}
p_i^{(r)}w_{q+i}^{(r)}
&=\sum_{r=1}^R\lambda_r
\sum_{\substack{0\le i<m\\q+i<n}}
p_i^{(r)}w_{q+i}^{(r)}
+\sum_{r=1}^R\lambda_r
\sum_{\substack{0\le i<m\\q+i\ge n}}p_i^{(r)}\\
&=(Bw)_q+1-c_q.
\end{aligned}
\]
The remaining finite covering inequality, with $q=n$, is automatic, since every index $q+i$ then lies in the fixed tail where $w_{q+i}^{(r)}=1$.

Probability normalization also gives
\begin{equation}\label{EQ:c=Atimes1}
    c_q=\sum_{r=1}^R\lambda_r
\sum_{\substack{0\le i<m\\q+i<n}}p_i^{(r)}
=(B\mathbf1_{Rn})_q.
\end{equation}
Thus those indexed by $0\le q<n$ are the only potentially nontrivial finite covering inequalities.
For fixed kernel vectors and mixing weights, minimizing $b$ is therefore equivalent to
\begin{equation}\label{eq:primal-qp}
\Phi_*:=\min\{w^{\mathsf T}Dw:Bw\ge c\},
\qquad
b_*=1+2\left(\frac{\Phi_*}{m}-L\right).
\end{equation}

\begin{lemma}\label{prop:finite-dual}
Let $G=BD^{-1}B^{\mathsf T}$. The dual of \eqref{eq:primal-qp} is
\begin{equation*}
\Phi_*=
\max_{y\ge0}\left\{c^{\mathsf T}y-\frac14y^{\mathsf T}Gy\right\}.
\end{equation*}
A primal-feasible stacked boundary vector $w$ and a dual-feasible vector $y\ge0$ are optimal if
\begin{equation}\label{EQ:KKT-condition}
2Dw=B^{\mathsf T}y,
\qquad
y_q\bigl((Bw)_q-c_q\bigr)=0\quad(0\le q<n).    
\end{equation}
The primal minimizer is unique.
\end{lemma}

\begin{proof}
For $y\ge0$, consider the Lagrangian
\[
\mathcal L(w,y):=w^{\mathsf T}Dw-y^{\mathsf T}(Bw-c).
\]
Since $D$ is positive definite, completing the square gives
\[
\begin{aligned}
\mathcal L(w,y)
&=\left(w-\frac12D^{-1}B^{\mathsf T}y\right)^{\mathsf T}
D\left(w-\frac12D^{-1}B^{\mathsf T}y\right)\\
&\quad+c^{\mathsf T}y-\frac14y^{\mathsf T}Gy,
\end{aligned}
\]
and hence
\[
\inf_w\mathcal L(w,y)
=c^{\mathsf T}y-\frac14y^{\mathsf T}Gy.
\]
This is the stated dual objective and, in particular, proves weak duality.

Recall from  \eqref{EQ:c=Atimes1} that $c=B\mathbf1_{Rn}$. If the row of $B$ indexed by $q$ is zero, then $c_q=0$, so the corresponding primal constraint is the identity $0\ge0$ and is redundant. Delete all such rows temporarily. Since $B$ is entrywise nonnegative and every remaining row is nonzero, we have $c_q=(B\mathbf1_{Rn})_q>0$ for every remaining index $q$. Hence, for any $M>1$,
\[
B(M\mathbf1_{Rn})=Mc>c.
\]
Thus $M\mathbf1_{Rn}$ is strictly feasible and the reduced problem satisfies
Slater's condition. Its feasible set is nonempty and closed. Moreover, if
$d_{\min}>0$ is the smallest diagonal entry of $D$, then
\[
w^{\mathsf T}Dw\ge d_{\min}\lVert w\rVert_2^2,
\]
so the objective is coercive. Consequently, every minimizing sequence is
bounded and has a convergent subsequence whose limit is feasible; hence the
primal minimum is attained. Slater's theorem now gives a dual maximizer and
equality of the primal and dual optimal values. Reinstating a deleted row
changes neither problem, because its primal constraint is $0\ge0$ and its
dual variable may be set to zero.

It remains to verify the stated optimality conditions explicitly. Suppose
that $w$ is primal feasible, $y\ge0$, and the two equations in the statement
hold. For any other primal feasible vector $v$, stationarity gives
\[
\begin{aligned}
v^{\mathsf T}Dv-w^{\mathsf T}Dw
&=2w^{\mathsf T}D(v-w)+(v-w)^{\mathsf T}D(v-w)\\
&=y^{\mathsf T}B(v-w)+(v-w)^{\mathsf T}D(v-w).
\end{aligned}
\]
By dual feasibility, primal feasibility, and complementary slackness,
\[
y^{\mathsf T}B(v-w)
=y^{\mathsf T}(Bv-c)-y^{\mathsf T}(Bw-c)\ge0.
\]
It follows that $v^{\mathsf T}Dv\ge w^{\mathsf T}Dw$, so $w$ is primal
optimal. Stationarity also says that $w$ minimizes the Lagrangian for this
$y$, while complementary slackness gives
\[
\inf_v\mathcal L(v,y)=\mathcal L(w,y)
=w^{\mathsf T}Dw-y^{\mathsf T}(Bw-c)=w^{\mathsf T}Dw.
\]
Thus $y$ is dual optimal as well. Finally, if $v\ne w$, then positive definiteness of $D$ gives $(v-w)^{\mathsf T}D(v-w)>0$ in the preceding comparison. Hence no second primal minimizer exists, and the primal minimizer is unique.
\end{proof}

\subsection{Automatic positivity}\label{subsec:positive-b}

We now verify that the condition $b>0$ required in
Theorem~\ref{thm:kernel-transfer} holds automatically for every feasible kernel--boundary system. For $0\le j<n$, define
\[
P_j^{(r)}:=\sum_{i=0}^{\min\{j,m-1\}}p_i^{(r)}.
\]
For every column $(r,j)$ of $B$,
\[
(B^{\mathsf T}\mathbf1_n)_{(r,j)}
=\sum_{q=0}^{n-1}B_{q,(r,j)}
=\lambda_r\sum_{i=0}^{\min\{j,m-1\}}p_i^{(r)}
=\lambda_rP_j^{(r)}.
\]
Consequently, using $c=B\mathbf1_{Rn}$ and the block-diagonal form of
$D^{-1}$, we have
\[
\mathbf1_n^{\mathsf T}c
=\sum_{r=1}^R\lambda_r\sum_{j=0}^{n-1}P_j^{(r)},
\qquad
\mathbf1_n^{\mathsf T}BD^{-1}B^{\mathsf T}\mathbf1_n
=\sum_{r=1}^R\lambda_r\sum_{j=0}^{n-1}
\bigl(P_j^{(r)}\bigr)^2.
\]
Since $2\mathbf1_n$ is dual feasible, evaluating the dual objective at this
point gives
\begin{equation}\label{EQ:feasbile-bound-2I}
 \Phi_*\ge
\sum_{r=1}^R\lambda_r\sum_{j=0}^{n-1}
\left(2P_j^{(r)}-\bigl(P_j^{(r)}\bigr)^2\right).   
\end{equation}

Set $f(t)=2t-t^2$. For $0\le j\le m-2$, symmetry gives
$P_j^{(r)}+P_{m-2-j}^{(r)}=1$. Hence, writing
$t=P_j^{(r)}\in[0,1]$, each paired contribution satisfies
\[
f\bigl(P_j^{(r)}\bigr)+f\bigl(P_{m-2-j}^{(r)}\bigr)
=f(t)+f(1-t)=1+2t(1-t)\ge1.
\]
\grammarrev{To make the pairing count explicit, put
$S_r:=\sum_{j=0}^{m-2}f(P_j^{(r)})$. Since
$j\mapsto m-2-j$ is a permutation of $\{0,\ldots,m-2\}$,}
\[
2S_r
=\sum_{j=0}^{m-2}
\left(f\bigl(P_j^{(r)}\bigr)
+f\bigl(P_{m-2-j}^{(r)}\bigr)\right)
\ge m-1.
\]
\grammarrev{Therefore,}
\[
\sum_{j=0}^{m-2}f\bigl(P_j^{(r)}\bigr)=S_r
\ge\frac{m-1}{2}.
\]
For $j\ge m-1$, probability normalization gives $P_j^{(r)}=1$ and hence
$f(P_j^{(r)})=1$. There are $n-m+1$ such indices, so, for every $r$,
\[
\sum_{j=0}^{n-1}f\bigl(P_j^{(r)}\bigr)
\ge\frac{m-1}{2}+n-m+1
=n-\frac{m-1}{2}.
\]
By \eqref{EQ:feasbile-bound-2I}, Lemma \ref{prop:finite-dual}  and the fact $\sum_{r=1}^R\lambda_r=1$, we conclude that 
\[
\Phi_*\ge n-\frac{m-1}{2},
\qquad
b_*\ge
1+2\left(\frac{n-(m-1)/2}{m}-L\right)
=\frac1m>0.
\]
This means 
\[
b\ge b_*\ge1/m.
\]

\subsection{Nonnegative least-squares form}\label{subsec:nnls}

To obtain a form convenient for numerical computation, write the dual
objective in Lemma \ref{prop:finite-dual}  as
\[
g(y):=c^{\mathsf T}y-\frac14y^{\mathsf T}BD^{-1}B^{\mathsf T}y.
\]
Here $c=B\mathbf1_{Rn}$ and
$\mathbf1_{Rn}^{\mathsf T}D\mathbf1_{Rn}
=n\sum_{r=1}^R\lambda_r=n$. Therefore, for every $y\ge0$, expanding the
squared norm gives
\[
\begin{aligned}
\left\|D^{1/2}\mathbf1_{Rn}
-\frac12D^{-1/2}B^{\mathsf T}y\right\|_2^2
&=\mathbf1_{Rn}^{\mathsf T}D\mathbf1_{Rn}
-\mathbf1_{Rn}^{\mathsf T}B^{\mathsf T}y
+\frac14y^{\mathsf T}BD^{-1}B^{\mathsf T}y\\
&=n-c^{\mathsf T}y
+\frac14y^{\mathsf T}BD^{-1}B^{\mathsf T}y\\
&=n-g(y).
\end{aligned}
\]
Since $n$ is independent of $y$, maximizing $g(y)$ over $y\ge0$ is
equivalent to the following nonnegative least-squares problem
\begin{equation}\label{nonnegative-least-squares-problem}
\min_{y\ge0}
\left\|D^{1/2}\mathbf1_{Rn}
-\frac12D^{-1/2}B^{\mathsf T}y\right\|_2^2.
\end{equation}

This is a standard nonnegative least-squares problem with target vector
$D^{1/2}\mathbf1_{Rn}$ and design matrix
$\frac12D^{-1/2}B^{\mathsf T}$. Once a dual maximizer $y_*$ has been
found, strong duality and the KKT conditions recover the unique primal minimizer as
\[
w_*=\frac12D^{-1}B^{\mathsf T}y_*.
\]
We used this formulation to solve the boundary problem with fixed kernel vectors and mixing weights during the exploratory computation.

\subsection{Exploratory outer search}\label{subsec:outer-search}

For each choice of kernel vectors and mixing weights, we formed $B$ and $D$, solved \eqref{nonnegative-least-squares-problem}, recovered $w$ from \eqref{EQ:KKT-condition}, and evaluated $\sqrt{ab}$.  We then varied the kernel-vector and mixing-weight parameters and re-solved the boundary quadratic program.  This is a nonconvex floating-point search for a candidate kernel--boundary system, not a proof of optimality.

Initially, we used free symmetric kernel vectors together with the uniform kernel vector.  Later kernel vectors were parametrized using six Fourier modes.  For $x_i=(i+1/2)/m-1/2$ and $\theta\in\mathbb R^6$, such a kernel vector has the form
\[
p_i(\theta)=
\frac{\displaystyle\exp\!\left(
\sum_{\ell=1}^6\theta_\ell\cos(2\pi\ell x_i)\right)}
{\displaystyle\sum_{j=0}^{m-1}\exp\!\left(
\sum_{\ell=1}^6\theta_\ell\cos(2\pi\ell x_j)\right)}
\qquad(0\le i<m).
\]
This automatically enforces positivity, normalization, and symmetry.  The
mixing weights were parametrized by
\[
\lambda_r=\frac{\exp(\alpha_r)}{\sum_{s=1}^R\exp(\alpha_s)}
\qquad(1\le r\le R),
\]
which automatically gives $\lambda_r>0$ and $\sum_{r=1}^R\lambda_r=1$.  Starting from the $R=2$ kernel--boundary system and adding one kernel vector at a time produced the continuation sequence in Table~\ref{tab:R-scan}.  Further implementation details belong to the exploratory search script and are not needed to verify the theorem.

At $m=32$ and $L=4$, the exploratory search script \nolinkurl{sidon_numerical_search.py} produced the best values encountered in our runs, as shown in Table~\ref{tab:R-scan}.
\begin{table}[ht]
\centering
\small
\renewcommand{\arraystretch}{1.08}

\begin{tabularx}{0.99\textwidth}{@{}
    >{\centering\arraybackslash}p{0.7cm}
    >{\centering\arraybackslash}X
    >{\centering\arraybackslash}p{3.7cm}
@{}}
\toprule
$R$ & search class or certificate & value of $\sqrt{ab}$ \\
\midrule
$1$ & free symmetric kernel vector
    & $0.9461473014$ \\

$2$ & one free symmetric kernel vector and one uniform kernel vector
    & $0.9450510577$ \\

$3$ & continued $R=2$ system plus one free symmetric kernel vector
    & $0.9437979301$ \\

$4$ & continued $R=3$ system plus the first six-mode kernel vector
    & $0.9436448861$ \\

$5$ & continued $R=4$ system plus the second six-mode kernel vector
    & $0.9435665926$ \\

$6$ & continued $R=5$ system plus the third six-mode kernel vector
    & $0.9435030631$ \\

$7$ & continued $R=6$ system plus the fourth six-mode kernel vector
    & $0.9434969530$ \\

$8$ & continued $R=7$ system plus the fifth six-mode kernel vector
    & $0.9434925901$ \\

$8$ & rounded exact rational certificate used in
      Theorem~\ref{thm:main}
    & $0.9434925907\ldots$ \\
\bottomrule
\end{tabularx}

\caption{Exploratory kernel-vector comparisons and the final value obtained from the exact rational certificate at
$m=32$ and $L=4$.}
\label{tab:R-scan}
\end{table}

The first eight entries in Table~\ref{tab:R-scan} are floating-point values found within the stated search classes, not certified global optima.  In particular, the $R=1$ row is numerical evidence rather than a lower bound for all $R=1$ kernel--boundary pairs.  The last row has a different logical status: its value is recomputed from the exact rational certificate used in Theorem~\ref{thm:main}.

\subsection{Rationalization and feasibility repair}
\label{subsec:rationalization}

Rounding the boundary vectors in a feasible kernel--boundary system can create small violations of the finite covering inequalities.  The following repair is useful independently of the present kernel--boundary system.  For rational rounded data, set
\[
M(q):=\sum_{r=1}^R\lambda_r\sum_{i=0}^{m-1}
p_i^{(r)}w_{q+i}^{(r)}
\qquad(0\le q\le n),
\]
with $w_j^{(r)}=1$ for $j\ge n$, and put
\[
\rho_q:=c_q=\sum_{r=1}^R\lambda_r
\sum_{\substack{0\le i<m\\q+i<n}}p_i^{(r)}
\qquad(0\le q<n).
\]

\begin{lemma}\label{lem:rational-repair}
\grammarrev{If the mixing weights and kernel vectors are nonnegative rational data with the required normalizations and every entry $w_j^{(r)}$, $0\le j<n$, is rational, then the kernel--boundary system obtained by replacing each $w^{(r)}$ with the boundary vector $\widetilde w^{(r)}$ defined by}
\[
\widetilde w_j^{(r)}=
\begin{cases}
w_j^{(r)}+\eta,&0\le j<n,\\
1,&j\ge n
\end{cases}
\]
is feasible whenever $\eta$ is rational and
\[
\eta\ge
\max_{\substack{0\le q<n\\ \rho_q>0}}
\frac{\max\{0,1-M(q)\}}{\rho_q}.
\]
\end{lemma}

\begin{proof}
For $0\le q<n$, let $M_\eta(q)$ denote the covering value after this replacement. Direct substitution gives
\[
M_\eta(q)=M(q)+\eta\rho_q.
\]
If $\rho_q=0$, nonnegativity implies that every kernel vector with $\lambda_r>0$ has zero mass on the indices with $q+i<n$. All of its mass then lies in the fixed tail, so $M(q)=1$.  Also $\rho_0=1$, and the finite covering inequality with $q=n$ remains equality because it contains only fixed-tail entries.  The displayed choice of $\eta$ therefore repairs all finite covering inequalities simultaneously.
\grammarrev{Under the stated rationality hypotheses, every $M(q)$ and $\rho_q$ is rational; since $\rho_0=1$, the maximum is nonempty and rational.
Thus a rational $\eta$ satisfying the displayed inequality always exists.}
\end{proof}

To obtain the final kernel--boundary system, we rounded the mixing weights and symmetric kernel vectors while preserving the mixing-weight normalization and the kernel-vector normalization, symmetry, and nonnegativity, re-solved the boundary quadratic program, \grammarrev{rounded every boundary entry to a rational number}, and applied Lemma~\ref{lem:rational-repair}. After this repair, \grammarrev{we reverified the normalization and nonnegativity of the mixing weights; the normalization, symmetry, and nonnegativity of the kernel vectors; and all finite covering inequalities. We also recomputed the exact values of $a$ and $b$ and performed the final comparison using rational arithmetic alone.}
The numerical optimizer is therefore a discovery tool only.
\section{Exact rational certificate and completion of the proof}\label{sec:certificate}

This section is logically independent of the floating-point values in
Table~\ref{tab:R-scan}.  The proof of Theorem~\ref{thm:main} uses only
Theorem~\ref{thm:kernel-transfer} and the exact assertions in
Lemma~\ref{cla:certificate} below.

We take
\[
R=8,\qquad m=32,\qquad L=4,
\]
and
\begin{equation}\label{eq:eight-lambdas}
\begin{aligned}
(\lambda_1,\ldots,\lambda_8)=10^{-8}(&39490874,8624912,135342,12911860,\\
&20451562,2832639,9217142,6335669).
\end{aligned}
\end{equation}
The indexing follows the continuation search: the first kernel vector is the original free symmetric kernel vector, while the uniform kernel vector was added as the second kernel vector at the $R=2$ stage. The remaining exact rational kernel vectors and all boundary vectors are provided in the machine-readable certificate accompanying the paper. 
The common denominator of the kernel-vector entries is $10^{12}$. Before the common repair, the rounded boundary entries have denominator $10^9$, and the common repair increment is $\eta=17/250000000000$.

The verifier checks all $129$ finite covering inequalities. The final finite covering inequality, with $q=128$, holds with equality, while the least positive slack occurs at $q=127$ and is
\[
\frac{4735171805469436153}
{6250000000000000000000000000000}>0.
\]
\begin{lemma}\label{cla:certificate}
The rational data in the ancillary certificate satisfy all $129$ finite covering
inequalities \eqref{eq:covering}. Moreover,
\begin{equation}\label{eq:exact-a}
a=\frac{497329054138522113993707809619}
{390625000000000000000000000000},
\end{equation}
\begin{equation}\label{eq:exact-b}
b=\frac{69918675237166718360455326217}
{100000000000000000000000000000},
\end{equation}
and
\begin{equation}\label{eq:constant-check}
ab<\left(\frac{9435}{10000}\right)^2.
\end{equation}
\end{lemma}

\begin{proof}
\grammarrev{The verifier used in this proof is \nolinkurl{sidon_certificate_8kernel.py}. A frozen copy of the exact file identified by the SHA--256 hash below is supplied with the ancillary files of this version.}
The accompanying verifier checks the normalization and nonnegativity of the mixing weights, the normalization, symmetry, and nonnegativity of the kernel vectors, all finite covering inequalities, the exact values of $a$ and $b$, and \eqref{eq:constant-check}. It also confirms the least positive slack displayed above. The certificate data are embedded in the verifier. The SHA--256 hash of \nolinkurl{sidon_certificate_8kernel.py} is
\begin{center}
\small\texttt{957a5afadd849ac4f97c2b71252abb5c796c2db3c91a608ab35097e3c49292a8}.
\end{center}
This completes the exact verification.
\end{proof}

\begin{proof}[Proof of Theorem~\ref{thm:main}]
Lemma~\ref{cla:certificate} supplies a feasible kernel--boundary system with $b>0$. 
Applying Theorem~\ref{thm:kernel-transfer} to this system and setting $\gamma_0=\sqrt{ab}$ gives
\[
F(N)\le N^{1/2}+\gamma_0N^{1/4}+O(1).
\]
Evaluating the exact product from \eqref{eq:exact-a} and \eqref{eq:exact-b} gives
\begin{equation*}
    \gamma_0=0.9434925907135450\ldots. \qedhere
\end{equation*}
\end{proof}

\section{Concluding remarks}

We give \grammarrev{an improved upper bound on $F(N)$} using an exact eight-pair rational certificate.  More generally, Theorem~\ref{thm:kernel-transfer}  separates the reusable mathematical mechanism from this particular certificate.  Any new feasible kernel--boundary system immediately gives a valid coefficient, without repeating the Sidon-set argument.  Proposition~\ref{prop:scalar-bottleneck} also isolates the source of improvement: averaging scalar-feasible kernel--boundary pairs cannot improve the best constituent coefficient, whereas joint feasibility permits complementary boundary deficits.

The finite framework can be reused with more kernel vectors, other parametrized families of kernel vectors, or different values of $m$ and $L$.  Its two requirements are an autocorrelation estimate controlling the weighted energy and the finite covering inequalities.  This suggests analogous programs for bounded-multiplicity difference problems or for constructions with suitable cross-correlation positivity, although such extensions would require separate analytic input.

Problem~\ref{Erdos-Turan-problem} remains the classical open problem in additive number theory.
Further optimization may reduce the coefficient of $N^{1/4}$, but the present method does not explain how to eliminate this term. Proving $F(N)\le N^{1/2}+O(1)$ will likely require a genuinely new structural idea.

\section*{Data and code availability}

The exact rational certificate and the exploratory search code are distributed separately.  The exact rational data and standard-library verifier are contained in \nolinkurl{sidon_certificate_8kernel.py}; the exploratory search code is \nolinkurl{sidon_numerical_search.py}.  Both files accompany the submission and are mirrored in the versioned repository
\[
\texttt{https://github.com/HbZhao1/sidon-vector-smoothing/tree/main}.
\]
The certificate file is identified by the SHA--256 hash in Lemma~\ref{cla:certificate}, so the checked artifact is unambiguous even if the repository later changes.  The verifier requires Python~3.9 or later and uses only the Python standard library. All verification checks use exact rational arithmetic; floating-point arithmetic is used only to print the decimal approximation of $\sqrt{ab}$.
Running
\[
\texttt{python3 sidon\_certificate\_8kernel.py}
\]
reproduces the exact checks asserted in Lemma~\ref{cla:certificate}.  

\section*{Declaration on the use of AI}

The proof strategy was developed by the authors, \grammarrev{inspired by the work of Carter, Hunter, and O'Bryant}~\cite{CHO}.  The authors used generative AI tools to help check and simplify routine technical computations. All \grammarrev{AI-assisted} suggestions and computations were independently checked and
verified by the authors.

\bibliographystyle{abbrv}
\bibliography{sidon}

\end{document}